%% file: journal.tex
\documentclass{amsart}
\usepackage{amssymb}
\newtheorem{theorem}{Theorem}[section]
\newtheorem{lemma}[theorem]{Lemma}
\newtheorem{prop}[theorem]{Proposition}
\newtheorem*{untheorem}{Theorem} 
\newtheorem{question}[theorem]{Question}

\theoremstyle{definition}
\newtheorem{definition}[theorem]{Definition}
\newtheorem{notation}[theorem]{Notation}
\newtheorem*{undef}{Definition}
\newtheorem{example}[theorem]{Example}

\newtheorem{cor}[theorem]{Corollary}
\newtheorem{remark}[theorem]{Remark}
\newtheorem{non-example}[theorem]{Non-Example}
\numberwithin{equation}{section}

\newcommand{\setlist}[1]{\left\{ #1 \right\}}
\newcommand{\setdesc}[2]{\left\{ #1 : #2 \right\}}

\begin{document}

\title{Decompositions of differential equations from medical imaging}

\author{Douglas Weathers}
\address{Department of Mathematics and Statistics, Coastal Carolina University, Conway, SC 29526}
\email{wweathers@coastal.edu}
\thanks{The first author was supported on this work, which first appeared as his master's thesis, by a Chase Distinguished Research Assistantship from the University of Maine and a teaching assistantship from its Department of Mathematics and Statistics.}

\author{Benjamin L. Weiss}
\address{Department of Mathematics and Statistics, University of Maine, Orono, ME 04473}
\email{benjamin.weiss@maine.edu}


\date{\today}


\begin{abstract}

Studying medical imaging, Peter Kuchment and Sergey Lvin encountered an countable family of differential identities for sine, cosine, and the exponential function. Specifically if for a smooth function $u$ and a complex number $\lambda$ the minimal differential equation $u' = \lambda u$ held, then $u$ satisfied all of their identities (i). If $u'' = \lambda^2 u$, then $u$ satisfied all odd-indexed identities (ii). They were unable to determine (iii) if there were some other pattern as well. We realize their $n$-th identity as a polynomial $f_{n,\lambda}(u)$ in the variable $u$ that turns out to have integer coefficients. We construct combinatorial relations on the coefficients to provide an alternate proof of one of Kuchment and Lvin's results. We also isolate the part of $f_{n,\lambda}(u)$ that is linear in the variable $u$ to answer (iii) negatively, and describe how the analysis of the linear polynomial may connect to the analysis of the whole polynomial.




\end{abstract}

\maketitle

\section{Introduction}

\input{ch1j}

\section{Preliminaries}

\input{ch2j}

\section{Combinatorial proof of the first K-L identity}
\input{ch3j}

\section{Rejection of new patterns}

\input{ch4j}



\bibliographystyle{amsplain}\bibliography{references}

\end{document}

%% file: ch1j.tex
\label{ch:one}

%
%
%
%
\noindent In the early 1990's, Kuchment and Lvin \cite{KL2} studied the mathematics behind Single Photon Emission Computed Tomography (SPECT), in which the patient is given some medication weakly labeled with $\gamma$-photons \cite{KL}. Discovering the distribution of the medication reduces to studying an attenuated Radon-type transform \cite{Kuc}, \cite{Dea}. Let $\lambda$ be a positive real number called the attenuation coefficient. Let $L$ be a line that is parametrized by the function $s\omega + t\omega^\perp$, where $s$ is a real number, $t$ is a varying parameter, $\omega$ is a unit vector in $\mathbb{R}^2$, and $\omega^\perp$ is $\omega$ rotated by ninety degrees clockwise. Then the attenuated Radon transform is given by

$$ R_\lambda \phi(L) = \int_{-\infty}^\infty \phi(s\omega + t\omega^\perp)e^{\lambda t} \; dt$$

\noindent where $\phi$ is a function from $\mathbb{R}^2$ to $\mathbb{R}$. 
In checking the range conditions required to recover the original function, the authors discovered that the following identity holds for all odd positive integers $n$ \cite{KL}:

$$ \sin^n x + \sum_{k=0}^{n-1} {{n}\choose{k}} \left( \prod_{m=0}^{n-k-1} \left( \frac{d}{dx} - \sin x + mi \right) \right) \sin^k x = 0.$$

\noindent More recently, the authors \cite{KL} generalized this identity in the following way:

\begin{definition}\label{K-L-poly} Let $n$ be a positive integer, $\lambda$ a complex number, $u = u(x)$ a smooth function, and $\partial = \frac{d}{dx}$. The $n$-th {\it Kuchment-Lvin (K-L) polynomial} parametrized by $\lambda$ is defined to be

$$ f_{n,\lambda}(u)  = u^n + \sum_{k=0}^{n-1} {{n}\choose{k}} \left(\prod_{m=0}^{n-k-1} \left( \partial - u + m\lambda \right) \right) u^k.$$
\end{definition}

\begin{notation} \label{Dd} In this paper $\alpha$ will always be an integer, and $u$ will be a smooth function. When $\alpha$ is non-negative, we will follow standard notation \cite{SvdP} that $u^{(\alpha)}$ will denote the function $u$ differentiated $\alpha$ times. We will also use $u^{(0)} = u$, and $u^{(1)} = u'$, and $u^{(2)} = u''$.\end{notation}

\noindent In \cite{KL}, the K-L polynomials were proven to vanish under the following conditions (Theorems 2 and 3 of \cite{KL} respectively):
	\begin{theorem}[The first Kuchment-Lvin identity]\label{kl1}If $u=u(x)$ is a smooth function and $\lambda$ a complex number such that $u' = \lambda u$, then $f_{n,\lambda}(u) = 0$ for all $n \ge 1$. \end{theorem}
	\begin{theorem}[The second Kuchment-Lvin identity]\label{kl2}If $u=u(x)$ is a smooth function and $\lambda$ a complex number such that $u'' = \lambda^2 u$, then  $f_{n,\lambda}(u) = 0$ for all odd $n$.\end{theorem}
\noindent The authors then posed, and left unanswered, the following question:

\begin{question}\label{more-patterns}Does some other pattern occur? Let $m$ be an integer such that $m \ge 3$. If there exists a complex number $\lambda$ such that $u^{(m)} = \lambda^m u$, then does every $m$-th K-L polynomial vanish?\end{question}

\noindent We answer a closely related question in the negative using combinatorial techniques of differential algebra. Namely we prove the following theorem:

\begin{untheorem}[\ref{thm5}] Let $u$ be a smooth function, $\lambda$ a nonzero complex number, and $m$ an integer such that $m \ge 3$ and $u^{(m)} = \lambda^m u$. Let $n$ be an integer such that $n \ge 2$. If $f^L_{n,\lambda}(u) = 0$, then $u' = \lambda u$ or $u'' = \lambda^2 u$.\end{untheorem}

\noindent Here $f^L_{n,\lambda}(u)$ is the sum of terms in $f_{n,\lambda}(u)$ which are degree $1$ in $u$ and its derivatives, called the linear part of $f_{n,\lambda}$. The above theorem says that for nonzero $\lambda$ and the linear part of the polynomials, the answer to Question \ref{more-patterns} is ``no'': all patterns that could be found are the patterns in Theorem \ref{kl1} and Theorem \ref{kl2}.

\begin{non-example}
Put $u = \sin x$. According to Theorem \ref{kl2}, $f_{3,i}(u) = 0$ since 3 is odd and $u'' = (i)^2u$. Let $\lambda = 1$. One may wonder if this and Theorem \ref{thm5} imply the contradiction $u' = \lambda u$ or $u^2 = \lambda^2 u$ because $u^{(4)} = \lambda^4 u$ and $u$ is a root of one of the K-L polynomials. This would not be correct, because the $\lambda$ in the equations $f_{n,\lambda}(u) = 0$ and $u^{(m)}=\lambda^m u$ must be the same number, and $1 \neq i$. 
\end{non-example}

\noindent One should also note that Theorem \ref{thm5} is neither a necessary nor sufficient statement towards answering Question \ref{more-patterns}. The Lindemann-Weierstra{\ss} Theorem \cite{Lang} states that if $a_1, \ldots, a_n$ are linearly independent over $\mathbb{Q}$, then $e^{a_1}, \ldots, e^{a_n}$ are algebraically independent over $\mathbb{Q}$. We would be able to use Theorem \ref{thm5} to answer Question \ref{more-patterns}, except that our set $a_1, \ldots, a_n$ are all the $n$-th roots of unity and therefore aren't linearly independent over $\mathbb{Q}$. 

This paper is organized as follows. In Section 2, we present definitions and lemmas that stand apart from the study of the K-L polynomials, but are necessary to the proofs given in later sections. Particularly, we discuss coefficients that arise in differentiating a function $u$ raised to some power (see Section 2.1), sums of products of integers (see Section 2.2), and a convolution involving a generalized binomial coefficient (see Section 2.3). In Section 3, we expand the definition of the K-L polynomials in Theorem \ref{expanded-form} and provide an alternate proof of the first Kuchment-Lvin identity by doing combinatorics on their coefficients. Having demonstrated the utility of the combinatorial approach, we apply those techniques in Section 4 to prove Theorem \ref{thm5} given above. To conclude, we explain the problems that arise in extending our analysis from the linear part to the whole polynomial and outline possible directions that future work on the K-L polynomials may take.

%% file: ch2j.tex
\label{ch:two}

\noindent In Theorem \ref{expanded-form}, we express the Kuchment-Lvin polynomials (Definition \ref{K-L-poly}) as a linear combination of powers of a complex number $\lambda$ attached to products of derivatives of a smooth function $u$. Studying the K-L polynomials using combinatorics will require us to look at coefficients that arise from the product rule of differentiation as well as those arising from adding together certain products of integers.

\subsection{Product rule coefficients}

\noindent The definitions that follow depend on the choice of a smooth function $u$; however, we will be treating these symbols formally, so we will suppress $u$ in much of the following notation. 

\begin{definition}\label{pi}Let $j$ and $\alpha$ be integers with $j \ge 1$ and $\alpha \ge 0$. A {\it differential product} $\pi$ with {\it degree} $j$ and {\it order} $\alpha$ is a product
$$ \pi = \pi_u = \prod_{m=1}^j u^{(\alpha_m)} $$
such that $\alpha_1, \ldots, \alpha_j$ are non-negative integers with $\alpha_1 + \cdots + \alpha_j = \alpha$. The collection of all $\pi$ with degree $j$ and order $\alpha$ is denoted $\Pi_{j,\alpha}=\Pi_{j,\alpha}^u$.
\end{definition}

\noindent We study how a differential product $\pi$ can arise from differentiating powers of the function $u$. In particular, we are interested in coefficients of $\pi$ that arise from this process.

\begin{notation}Let $j$ and $\alpha$ be integers with $j$ positive. Put
$$ Z_{j,\alpha} = \setdesc{(\beta_1, \ldots, \beta_j)}{0 \le \beta_m  \in \mathbf{Z}, \; \beta_1 + \cdots + \beta_j = \alpha},$$ and notice that $Z_{j,\alpha} = \varnothing$ when $\alpha < 0$.
\end{notation}

\begin{definition} Let $\beta = (\beta_1, \ldots, \beta_j) \in Z_{j,\alpha}$ with $\alpha \ge 0$. Let $u(x)$ be a smooth function and $\partial = \frac{d}{dx}$, as above. We define the {\it differential word}, $w(\beta)$, of $\beta$ with respect to $u$ inductively, as follows:\begin{itemize}
\item for $j = 1$, we let $w(\beta_1) := \partial^{\beta_1}u;$ and
\item for $j > 1$ we define $w(\beta_1, \ldots, \beta_j) := \partial^{\beta_j}(uw(\beta_1, \ldots, \beta_{j-1}))$.
\end{itemize}
Hence $w(\beta_1,\ldots,\beta_j) = \partial^{\beta_j}(u\partial^{\beta_{j-1}}(u\cdots\partial^{\beta_1}u)\cdots)$, which we will write as $$w(\beta) = \partial^{\beta_j}u\partial^{\beta_{j-1}}u\cdots\partial^{\beta_1}u,$$ for brevity.
\end{definition}


\begin{example} Let $\beta = (0,1,1,1)$. Then
$$ w(\beta) = \partial u\partial u\partial u^2 =  8u(u')^3 + 14u^2u'u'' + 2u^3u^{(3)},$$
as is easily verified.
\end{example}

\noindent In the example above, the fact that $\beta$ has a leading zero entry allows us to simplify the calculation. We will be classifying coefficients that arise from examples with the first $k$ entries of $\beta$ equal to $0$ for general integers $k$. With this in mind, we introduce the following notation:

\begin{notation}\label{integer-tuples}Let $j$ and $k$ be positive integers such that $k \le j$. Let $\alpha$ be a non-negative integer. Denote the set of all members of $Z_{j,\alpha}$ whose first $k-1$ entries are zero by
$$ Z_{j,\alpha,k} = \setdesc{\beta \in Z_{j,\alpha}}{\beta_1, \ldots, \beta_{k-1} = 0}.$$
\end{notation}
\noindent Notice we have that $Z_{j,\alpha,1} = Z_{j,\alpha}$, and that $\Pi_{j,\alpha}$ and $Z_{j,\alpha}$ are defined so that all monomials arising from $w(\beta)$ for $\beta \in Z_{j,\alpha}$ are elements of $\Pi_{j,\alpha}$.

\begin{definition} \label{product-rule-coeff} Let $j$ and $k$ be positive integers such that $k \le j$. Let $\alpha$ be a non-negative integer, let $\pi \in \Pi_{j,\alpha}$, and let $\beta \in Z_{j,\alpha,k}$. The {\it product rule coefficient} associated to $\pi$ with respect to $\beta$ is the number $P_{\beta,\pi}$ such that
$$ w(\beta) = \sum_{\pi \in \Pi_{j,\alpha}} P_{\beta,\pi} \pi.$$\end{definition}

\begin{example} In the previous example, notice that $\beta = (0,1,1,1) \in Z_{4,3,2}$ and
$$P_{\beta,u(u')^3} = 8,~~~P_{\beta,uu'u''} = 14,~~~\mbox{and}~~P_{B,u^3u^{(3)}} = 2.$$
\end{example}

\noindent The first K-L identity (Theorem \ref{kl1}) supposes $u$ satisfies the first-order linear differential equation $u' = \lambda u$ for a complex number $\lambda$. We now present definitions and lemmas that arise while studying $w(\beta)$ when $u' = \lambda u$.

\begin{lemma} \label{pi=lambda^alphau^j}Let $j \ge 1$ and $\alpha \ge 0$ be integers. Suppose there exists a complex number $\lambda$ such that $u' = \lambda u$. If $\pi \in \Pi_{j,\alpha}$, then $\pi = \lambda^{\alpha}u^j$.\end{lemma}

\begin{proof}First, observe that if $\alpha_m$ is a positive integer and $u' = \lambda u$, then successive differentiation gives $u^{(\alpha_m)} = \lambda^{\alpha_m}u$. Next, recall $\pi$ is a product of derivatives of $u$ whose orders $\alpha_m$ sum to $\alpha$. Then
$$ \pi = \prod_{m=1}^j u^{(\alpha_m)} = \prod_{m=1}^j \lambda^{\alpha_m}u
= \lambda^{\alpha_1 + \cdots + \alpha_j} \prod_{m=1}^j u = \lambda^\alpha u^j$$
as claimed.\end{proof}

\noindent The preceding lemma implies that if $u' = \lambda u$, then different differential products $\pi$ can be added together.

\begin{example}Let $\beta = (0,1,1,1)$. If $\lambda$ is a complex number and $u' = \lambda u$, then
\begin{align*} w(\beta) = \partial u\partial u\partial u^2 &= 8(u')^3 + 14u^2u'u'' + 2u^3 u^{(3)} \\
& = 8\lambda^3 u^3 + 14 \lambda^3 u^3 + 2 \lambda^3 u^3 = 24\lambda^3 u^3
\end{align*}\end{example}

\begin{definition}\label{density} Let $j$ and $k$ be positive integers such that $k \le j$ and let $\alpha$ be a non-negative integer. Let $\beta \in Z_{j,\alpha,k}$. The {\it density} of $\beta$ is
$$|\beta| = \sum_{\pi \in \Pi_{j,\alpha}} P_{\beta,\pi}.$$\end{definition}

\begin{lemma}Let $j$ be a positive integer, and $\alpha$ a non-negative integer and let $\beta \in Z_{j,\alpha}$. If there exists a complex number $\lambda$ such that $u' = \lambda u$, then
$$ w(\beta) = \left( \prod_{m=1}^j m^{\beta_m} \right) \lambda^\alpha u^j.$$\end{lemma}

\begin{proof} We will let $k$ be an integer with $1 \le k \le j$, and induct on $j-k$ while proving the theorem for $\beta \in Z_{j,\alpha,k}$. Since $Z_{j,\alpha,1} = Z_{j,\alpha}$ this will prove the lemma.

Let $j = k$ so that $\beta = (0, \ldots, 0, \beta_k)$ and $\beta_k = \beta_j = \alpha$. Then
$$ w(\beta) = \partial^{\beta_j} u^j.$$
To show $w(\beta) = j^{\beta_j} \lambda^\alpha u^j$, we must perform a second induction on $\beta_j$. We establish the base case:
$$ \partial^{\beta_j} = \partial u^j = ju^{j-1} u' = j\lambda u^j.$$
The rest follows easily by induction on $j - k$.
\end{proof}

\begin{cor}\label{density-as-product}Let $j$ be a positive integer, and $\alpha$ a non-negative integer,  and $\beta \in Z_{j,\alpha}$. Then
$$ |\beta| =  \prod_{m=1}^j m^{\beta_m}.$$\end{cor}

\begin{example} The previous example showed that if $\beta = (0,1,1,1)$, then
$$ |\beta| = 2^1\cdot 3^1 \cdot 4^1 = 24.$$ Table \ref{density-table} gives all $\beta \in Z_{4,3,2}$ and their corresponding densities. The last line of the table is the sum of the densities, which is a key component in the study of the K-L polynomials:\end{example}

\begin{table}
\begin{center}
\begin{tabular}{l|r}
$\beta \in Z_{4,3,2}$ & $\displaystyle\sum_{\pi \in \Pi_{4,3}} P_{\beta,\pi}$\\
\hline
$(0,0,0,3)$ & 64 \\
$(0,0,1,2)$ & 48 \\
$(0,0,2,1)$ & 36 \\
$(0,0,3,0)$ & 27 \\
$(0,1,1,1)$ & 24 \\
$(0,1,2,0)$ & 18 \\
$(0,1,0,2)$ & 32 \\
$(0,2,1,0)$ & 12 \\
$(0,2,0,1)$ & 16 \\
$(0,3,0,0)$ & 8 \\
\hline
$W(4,3,2)$ & 285
\end{tabular}
\caption{\label{density-table}The sum of densities corresponding to $\beta \in Z_{4,3,2}$.}
\end{center}
\end{table}

\begin{definition}\label{weight} Let $k,j$, and $\alpha$ be integers with $j$ positive and $0\le k \le j$. The {\it weight} of $j$, $\alpha$, and $k$ is $$W(j,\alpha,k) = \begin{cases} \displaystyle\sum_{\beta \in Z_{j,\alpha, k}} |\beta| & k \ge 1, \alpha \ge 0 \\ 
W(j,\alpha,1) & k = 0, \alpha \ge 0 \\ 
0 & \alpha < 0. \end{cases}$$\end{definition}

\begin{remark} \label{weight-remark}The definition of $W(j,\alpha,k)$ represents the sum of all possible product rule coefficients $P_{\beta,\pi}$ arising from
$$ \partial^{\beta_j}u \cdots \partial^{\beta_k}u^k $$
for all possible choices of $\beta_k, \ldots, \beta_j$. If $k \ge 1$ and $\alpha \ge 0$, then
$$ W(j,\alpha,k) = \sum_{\beta \in Z_{j,\alpha,k}} \sum_{\pi \in \Pi_{j,\alpha}} P_{\beta,\pi} = \sum_{\beta \in Z_{j,\alpha,k}} |\beta| .$$
If we take $k=0$, then the definition of $Z_{j,\alpha,k}$ is undefined. However, the setting $W(j,\alpha,0) = W(j,\alpha,1)$ is consistent with the above interpretation: Suppose we instead defined $\beta$ to be a $(j+1)$-tuple with first entry $\beta_0$. Then we look at words of the form
$$ \partial^{\beta_j}u \cdots \partial^{\beta_1}u\partial^{\beta_0} = \partial^{\beta_j}u \cdots \partial^{\beta_1}u\partial^{\beta_0}1.$$
Thus nothing is lost by restricting the discussion to $j$-tuples $\beta$ and setting $W(j,\alpha,0) = W(j,\alpha,1)$. Furthermore, the proof of Lemma \ref{C*j=0} relies very critically on the fact that in Lemma \ref{weight-as-sum}, the index $m$ cannot be less than one.

Until this point, we have required $\alpha$ to be non-negative. In the proof of Lemma \ref{C*j=0}, we re-index a sum in such a way $\alpha$ may be negative. If $\alpha < 0$, then the set $Z_{j,\alpha,k}$ is empty; there are no tuples of non-negative integers whose entries sum to $\alpha$. In that case, the sum is empty and is defined to be zero.\end{remark}

\begin{example}Table \ref{density-table} gives that the sum of all the densities corresponding to $\beta \in Z_{4,3,2}$---{\it i.e.}, the weight of 4, 3, and 2---is 285. Using Corollary \ref{density-as-product} and Definition \ref{weight}, we will arrive at the same value for $W(4,3,2)$.
\begin{align*} W(4,3,2) & = \; \sum_{\beta \in Z_{4,3,2}} 2^{{\beta_2}} 3^{\beta_3} 4^{\beta_4} \\
& = \; \sum_{{\beta_2}=0}^3 2^{\beta_2} \sum_{{\beta_3} = 0}^{3-{\beta_2}} 3^{\beta_3} 4^{3 - {\beta_2} - {\beta_3}}.\end{align*}
Simplifying these finite geometric series, we have that
\begin{equation}  W(4,3,2)  = \; 4 \frac{4^{3+1}-2^{3+1}}{2} - 3 \frac{3^{3+1}-2^{3+1}}{1} = 285.
\end{equation}
The above line can be rewritten as follows:
$$ W(4,3,2) = \frac{4^2}{2!}\cdot 1 \cdot 4^3 - \frac{3^1}{1!}\cdot 3 \cdot 3^3 + \frac{2^0}{0!} \cdot 2 \cdot 2^3,$$
which takes into account the way that one would arrive at a linear combination of $4^3$, $3^3$, and $2^3$ when computing $W(4,3,2)$. Expanding $W(j,\alpha,k)$ in this way is computationally useful in the proof of Lemma \ref{C*j=0}. Remarkably, the coefficients on $4^\alpha$, $3^\alpha$, and $2^\alpha$ in these expansions do not change with $\alpha$; an alternative proof of this can be found using the generating function of partial fraction addition. Thanks to Julian Rosen for bringing this to our attention. A similar calculation shows that
$$ W(4,5,2) = 6069 =  \frac{4^2}{2!}\cdot 1 \cdot 4^5 - \frac{3^1}{1!}\cdot 3 \cdot 3^5 + \frac{2^0}{0!} \cdot 2 \cdot 2^5.$$
\end{example}

\begin{lemma}\label{weight-as-sum}Let $j$ and $k$ be positive integers such that $k \le j$. Let $\alpha$ be a non-negative integer. For $k \le m \le j$, there exist rational numbers $A_{m,j}$, independent of $\alpha$, such that
$$ W(j,\alpha,k) = \sum_{m=k}^j \frac{m^{m-k}}{(m-k)!} A_{m,j} m^\alpha.$$\end{lemma}

\begin{proof}
Recall from Definition \ref{weight} that
$$ W(j,\alpha,k) = \sum_{\beta \in Z_{j,k}^\alpha} |\beta|.$$
By Corollary \ref{density-as-product},
$$ |\beta| =  \prod_{m=k}^j m^{\beta_m},$$
so
$$ W(j,\alpha,k) = \sum_{\beta \in Z_{j,\alpha,k}}  \prod_{m=k}^j m^{\beta_m}.$$
We can express this sum in the following way:
\begin{align*}W(j,\alpha,k) &= \sum_{\beta \in Z_{j,\alpha,k}}  \prod_{m=k}^j m^{\beta_m} \\
& = \sum_{\beta_k + \cdots + \beta_j = \alpha}  \prod_{m=k}^j m^{\beta_m} \\
&= \sum_{\beta_k = 0}^\alpha k^{\beta_k} \sum_{\beta_{k+1} + \cdots + \beta_j = \alpha-\beta_k}  \prod_{m=k+1}^j m^{\beta_m} \\
& = \sum_{\beta_k = 0}^\alpha k^{\beta_k} W(j,\alpha - \beta_k, k+1).
\end{align*}
This suggests induction on $j-k$ as the correct strategy. First, consider the case where $j=k$:
$$W(k,\alpha,k) = \sum_{\beta_k = \alpha} k^{\beta_k} = k^\alpha.$$
The base case is satisfied taking $A_{k,k} = 1$. Next, let $j > k$ and suppose the claim is true for $W(j, \alpha-\beta_k, k+1)$. Then
\begin{align*}W(j,\alpha,k) & = \sum_{\beta_k=0}^\alpha k^{\beta_k} W(j, \alpha-\beta_k, k+1) \\
& = \sum_{\beta_k=0}^\alpha k^{\beta_k} \sum_{m=k+1}^j \frac{m^{m-k-1}}{(m-k-1)!} A_{m,j} m^{\alpha - \beta_k} \\
& = \sum_{m=k+1}^j \frac{m^{m-k-1}}{(m-k-1)!} A_{m,j}\sum_{\beta_k=0}^\alpha k^{\beta_k}m^{\alpha-\beta_k} \\
&= \sum_{m=k+1}^j \frac{m^{m-k-1}}{(m-k-1)!} A_{m,j} m^\alpha \sum_{\beta_k=0}^\alpha \left( \frac{k}{m} \right)^{\beta_k}.\end{align*}
The sum
$$\sum_{\beta_k=0}^\alpha \left( \frac{k}{m} \right)^{\beta_k}$$
is a finite geometric series with value
$$ \sum_{\beta_k=0}^\alpha \left( \frac{k}{m} \right)^{\beta_k} = \frac{1 - (k/m)^{\alpha+1}}{1 - k/m} = \frac{m-k^{\alpha+1}/m^\alpha}{m-k}.$$
Therefore,
\begin{align*}& \sum_{m=k+1}^j \frac{m^{m-k-1}}{(m-k-1)!} A_{m,j} m^\alpha\frac{m-k^{\alpha+1}/m^\alpha}{m-k} \\ = & \sum_{m=k+1}^j \frac{m^{m-k-1}}{(m-k-1)!} A_{m,j} \frac{m^{\alpha+1} - k^{\alpha + 1}}{m-k} \\
= & -\sum_{m=k+1}^j \frac{m^{m-k-1}}{(m-k)!} A_{m,j} k^{\alpha+1} + \sum_{m=k+1}^j \frac{m^{m-k}}{(m-k)!} A_{m,j} m^{\alpha}.
\end{align*}
Define
$$ A_{k,j} = -k\sum_{m=k+1}^j \frac{m^{m-k-1}}{(m-k)!} A_{m,j}$$
so that
\begin{align*} W(j,\alpha,k) & = A_{k,j} k^\alpha + \sum_{m=k+1}^j \frac{m^{m-k}}{(m-k)!} A_{m,j} m^\alpha  \\
& = \frac{k^{k-k}}{(k-k)!} A_{k,j} k^\alpha + \sum_{m=k+1}^j \frac{m^{m-k}}{(m-k)!} A_{m,j} m^\alpha \\
& = \sum_{m=k}^j \frac{m^{m-k}}{(m-k)!} A_{m,j} m^\alpha \end{align*}
as needed.\end{proof}

\subsection{Sums of products of integers}

\noindent Let $n$ be a positive integer and $\alpha$ be a non-negative integer. Another key ingredient in the combinatorial proof for Kuchment and Lvin's first result is a sum of all products of $\alpha$ integers between 1 and $n$. 

\begin{definition}\label{product-sum}Let $n > 0$ and $\alpha$ be integers. An $\alpha$-{\it product of }$n$ is a product of $\alpha$ distinct integers between 1 and $n$. Let $A$ range over subsets of $\setlist{1,\ldots,n}$. The sum of all such products is given by
$$ S(n,\alpha) = \begin{cases}{\displaystyle\sum_{|A| = \alpha} \prod_{a \in A}} a & n \ge 1, 1 \le \alpha \le n\\
1 & \alpha = 0 \\
0 & \text{otherwise.} \end{cases}$$\end{definition}

\noindent These numbers can be easily understood when realized as the coefficients of the polynomial $g_n(z)$ defined below. 

\begin{lemma} \label{g(z)} The polynomial
$$ g_n(z) = \prod_{a=1}^n (1 + az)$$
is a generating function for the numbers $S(n,\alpha)$. In other words,
$$ g_n(z) = \sum_{\alpha=0}^n S(n,\alpha)z^\alpha.$$\end{lemma}

\begin{proof} Expand the product in the definition of $g_n(z)$ into a sum of powers of $z$. To obtain any $z^\alpha$ in the expansion of this product, we must multiply together $\alpha$ distinct integers $a$ such that $1 \le a \le n$. The coefficient on $z^\alpha$ will be the sum of all these products, which by definition is exactly $S(n,\alpha)$.\end{proof}

\begin{remark} Note additionally that the only constant term is obtained by multiplying $1$ by itself $n$ times; $S(n,0) = 1$. This explains the definition for $S(n,\alpha)$.\end{remark}

\begin{lemma}\label{product-sum-factorial}Let $n$ be a positive integer and $\alpha$ a non-negative integer. If $1 \le m \le n$, then
\begin{equation}\sum_{\alpha = 0}^n m^{n-\alpha+1}S(n,\alpha) = \frac{(n+m)!}{(m-1)!}.\end{equation}\end{lemma}

\begin{proof}Notice that
$$ m^{n+1}g_n\left(m^{-1}\right) = m^{n+1} \sum_{\alpha=0}^n m^{-\alpha} S(n,\alpha) = \sum_{\alpha=0}^n m^{n-\alpha+1} S(n,\alpha).$$
Using the product form of $g_n(z)$ we have
$$ m^{n+1}g_n\left(m^{-1}\right) = m^{n+1} \prod_{a=1}^n \left( 1 + \frac{a}{m} \right).$$
We can multiply each factor of the product by one of the copies of $m$ to get
$$ m \prod_{a=1}^n (m + a)$$
which is exactly
$$ m(m+1)(m+2) \cdots (m+n) = \frac{(n+m)!}{(m-1)!}$$
as claimed.
\end{proof}

\subsection{A convolution identity}

\noindent We prove one last proposition before addressing the K-L polynomials.

\begin{definition}Let $\gamma$ be a complex number and $q$ be a non-negative integer. The {\it generalized binomial coefficient} of $\gamma$ and $q$ is
$$ {{\gamma}\choose{q}} = \frac{(\gamma)(\gamma-1)\cdots (\gamma - q + 1)}{q!}.$$\end{definition}

\begin{prop} \label{negative-binomial}Let $n$, $m$, and $k$ be integers such that $n$ and $m-k$ are non-negative. Then
$$ {{-n}\choose{m-k}} = (-1)^{m-k}{{n+m-k-1}\choose{m-k}}.$$\end{prop}
\noindent We thank David Bradley for pointing us towards the generalized binomial coefficient and its use in noticing the coefficients of the constant function in Equation \eqref{coeffs-of-one}.

\begin{proof}By the definition of the generalized binomial coefficient,
$$ {{-n}\choose{m-k}} = \frac{(-n)(-n-1)\cdots (-n-(m-k-1))}{(m-k)!}.$$
We may take $-1$ out of each of the $m-k$ factors in the numerator to get
$$ {{-n}\choose{m-k}} = (-1)^{m-k}\frac{(n)(n+1)\cdots (n+m-k-1)}{(m-k)!} = (-1)^{m-k} {{n+m-k-1}\choose{m-k}}$$
as desired.\end{proof}

\noindent The remainder of this paper will use the more general combinatorial results from this chapter to study the Kuchment-Lvin polynomials defined in the introduction. We begin by expressing the K-L polynomials as linear combinations of differential products of their argument $u$ and providing an alternate proof of the first Kuchment-Lvin identity.

%% file: ch3j.tex
\label{ch:three}

\noindent In all that follows let $\lambda$ be a complex number, $n$ be a positive integer, and $u=u(x)$ be a smooth function. In this chapter we provide an alternate proof of the first Kuchment-Lvin identity. We note that unlike the original proof, we do not need to deal with the case of $\lambda = 0$ separately.

\begin{untheorem}[\ref{kl1}] If $u$ is a smooth function of $x$ and $\lambda$ is a complex number such that $u' = \lambda u$, then $f_{n,\lambda}(u) = 0$ for all $n \ge 1$. \end{untheorem}

\subsection{The Kuchment-Lvin polynomials}

\noindent Recall Definition \ref{K-L-poly}, reproduced here:

\begin{undef}[\ref{K-L-poly}] Let $n$ be a positive integer, $\lambda$ be a complex number, $u = u(x)$ be a smooth function, and $\partial = \frac{d}{dx}$. The $n$-th {\it Kuchment-Lvin (K-L) polynomial} parametrized by $\lambda$ is defined to be
$$ f_{n,\lambda}(u)  = u^n + \sum_{k=0}^{n-1} {{n}\choose{k}} \left(\prod_{m=0}^{n-k-1} \left( \partial - u + m\lambda \right) \right) u^k.$$
\end{undef}

\noindent Note that with these particular differential operators, the product of operators behaves associatively and commutatively.

\begin{example}Compute the $k=1$ term of $f_{3,\lambda}(u)$,
$$ {{3}\choose{1}}(\partial-u)(\partial-u+\lambda)u.$$
First, $(\partial - u + \lambda)$ acts on $u$:
$$ (\partial-u+\lambda)u = u' - u^2 + \lambda u.$$
Next, $(\partial - u)$ acts on the result:
$$ (\partial - u)(u' - u^2 + \lambda u) = u'' - u'u - 2u'u + u^3 + \lambda u' - \lambda u^2.$$
Finally, we add and multiply by ${{3}\choose{1}} = 3$.
$$ 3u'' - 9uu' + 3u^3 + 3\lambda u' - 3 \lambda u^2.$$\end{example}

\begin{example} One can similarly compute that the $k=0$ term of $f_{3,\lambda}(u)$ is
\begin{equation*} 
 -u'' + 3uu' -u^3 - 3\lambda u' + 3\lambda u^2 - 2\lambda^2 u\end{equation*}
and that the $k=2$ term is
$$ {{3}\choose{2}}(\partial-u)u^2 = 6uu' - 3u^3.$$
We add these to $u^3$ to get that:
$$ f_{3,\lambda}(u) = u^3 + \sum_{k=0}^2 {{3}\choose{k}} \left( \prod_{m=0}^2 (\partial - u + m\lambda) \right)u^k = 2u'' - 2\lambda^2u.$$
Observe that the coefficients of $f_{3,\lambda}(u)$---by which we mean the integers attached to $\lambda$ times some product of derivatives of $u$ (see Definition \ref{coefficient})---sum to zero. In fact, if we assume that $u' = \lambda u$, then $u'' = \lambda^2u$. Also note that
$$ f_{3,\lambda}(u) = \lambda^2u(2-2) = 0,$$
which demonstrates the first and second K-L identities (see Theorems \ref{kl1} and \ref{kl2}, respectively).\end{example}

\noindent Recall from Definition \ref{pi} that $\Pi_{j,\alpha}$ is the set of all differential products
$$ \pi = \prod_{m=1}^j u^{(\alpha_m)} $$ of derivatives of $u$ with degree $j$ and order $\alpha = \alpha_1 + \cdots + \alpha_j$.
\begin{theorem}\label{expanded-form}Let $n$ be a positive integer, $\lambda$ be a complex number, and $f_{n,\lambda}(u)$ be a K-L polynomial. Then 
$$ f_{n,\lambda}(u) = \sum_{j=1}^n \sum_{\alpha = 0}^{n-j} \lambda^{n-j-\alpha} \sum_{\pi \in \Pi_{j,\alpha}} C_\pi \pi, \qquad C_\pi \in \mathbf{Z}.$$\end{theorem}

\begin{definition}\label{coefficient}The integer $C_\pi$ in the above expression is called a {\it coefficient} of the K-L polynomial.\end{definition}

\noindent The proof of Theorem \ref{expanded-form} requires the following lemmas.

\begin{lemma}\label{kth-term-expansion} Let $k$ be an integer such that $0 \le k \le n-1$. Then
\begin{align*}
& {{n}\choose{k}} \left( \prod_{m=0}^{n-k-1} (\partial - u + m\lambda) \right) u^k \\
= & {{n}\choose{k}} \sum_{j=k}^n \sum_{\alpha =0}^{n-j} \lambda^{n-j-\alpha} S(n-k-1,n-j-\alpha) \sum_{\beta \in Z_{j,\alpha,k}} \sum_{\pi \in \Pi_{j,\alpha}} P_{\beta,\pi} \pi.  
\end{align*}\end{lemma}

\begin{proof}
Let $j$ and $\alpha$ be integers such that $0 \le j-k \le n-k$ and $0 \le \alpha \le n-j$. A term in the expansion of 
$$ \left( \prod_{m=0}^{n-k-1} (\partial - u + m\lambda) \right)u^k \label{kth-term} $$
arises from multiplying by $-u$ a total of $j-k$ times and differentiating $\alpha$ times. The remaining
$$ n-k - (j-k) - \alpha = n-j-\alpha $$
operators act as multiplication by $m\lambda$. We add over all possible choices of products of $n-j-\alpha$ integers from $\setlist{1, \ldots, n-k-1}$. Therefore, there exists a $\beta = (0,\ldots,0,\beta_k,\ldots,\beta_j) \in Z_{j,\alpha,k}$ (see Notation \ref{integer-tuples}) such that the result is
$$ (-1)^{j-k} \lambda^{n-j-\alpha} S(n-k-1,n-j-\alpha) \partial^{\beta_j} u \partial^{\beta_{j-1}} u \cdots \partial^{\beta_k} u^k$$
where $S(n-k-1,n-j-\alpha)$ is as defined in Definition \ref{product-sum}.
After expanding $\partial^{\beta_j} u \partial^{\beta_{j-1}} u \cdots \partial^{\beta_k} u^k$ as in Definition \ref{product-rule-coeff}, the above becomes
$$ (-1)^{j-k} \lambda^{n-j-\alpha} S(n-k-1,n-j-\alpha) \sum_{\pi \in \Pi_{j,\alpha}} P_{\beta,\pi} \pi.$$
Sum over all possible arrangements of $\partial$ and $u$---in other words, all $\beta \in Z_{j,\alpha,k}$---and all possible choices of $j$ and $\alpha$ to obtain
$$ \sum_{j=k}^n (-1)^{j-k} \sum_{\alpha=0}^{n-j} \lambda^{n-j-\alpha} S(n-k-1,n-j-\alpha) \sum_{\beta \in Z_{j,\alpha,k}} \sum_{\pi \in \Pi_{j,\alpha}} P_{\beta,\pi} \pi.$$
Multiplying by ${{n}\choose{k}}$ completes the proof.
\end{proof}

\begin{lemma}\label{no-constants} There are no constant terms in the expansion of $f_{n,\lambda}(u)$. \end{lemma}

\begin{proof}
Observe that
$$ \prod_{m=0}^{n-k-1}(\partial - u + m\lambda) $$
does not decrease the degree of its argument. Therefore, a constant term may only arise from the $k=0$ term,
$$ (\partial - u)(\partial - u + \lambda) \cdots (\partial - u + (n-1)\lambda)1.$$
We seek a term of degree zero. The right-most $n-1$ operators must act as multiplication by $m\lambda$ because multiplication by $-u$ will increase the degree and differentiation will annihilate the constant. Therefore after the first $n-1$ operators we obtain
$$ (\partial -u)(n-1)!\lambda^{n-1},$$
which will either be annihilated by $\partial$ or become $-(n-1)!\lambda^{n-1}u$, a term of degree one.
\end{proof}

\begin{proof}[Proof of Theorem \ref{expanded-form}] 
Recall that
$$ f_{n,\lambda}(u) = u^n + \sum_{k=0}^{n-1} {{n}\choose{k}} \left( \prod_{m=0}^{n-k-1} (\partial - u + m\lambda) \right)u^k.$$
By Lemma \ref{kth-term-expansion},
$$ f_{n,\lambda}(u) = u^n + \sum_{k=0}^{n-1} {{n}\choose{k}} \sum_{j=k}^n (-1)^{j-k} \sum_{\alpha =0}^{n-j} \lambda^{n-j-\alpha} S(n-k-1,n-j-\alpha) \sum_{\beta \in Z_{j,\alpha,k}} \sum_{\pi \in \Pi_{j,\alpha}} P_{\beta,\pi} \pi.$$
Trivially,
$$ u^n = {{n}\choose{n}} \sum_{j=n}^n \sum_{\alpha=0}^{n-n} \lambda^{n-n-0} S(n-n-1,n-n-0) \sum_{\beta \in Z_{n,0}^n} \sum_{\pi \in \Pi_{n,0}} P_{\beta,\pi} \pi,$$
so we can write
$$ f_{n,\lambda}(u) = \sum_{k=0}^n {{n}\choose{k}} \sum_{j=k}^n (-1)^{j-k} \sum_{\alpha =0}^{n-j} \lambda^{n-j-\alpha} S(n-k-1,n-j-\alpha) \sum_{\beta \in Z_{j,\alpha,k}} \sum_{\pi \in \Pi_{j,\alpha}} P_{\beta,\pi} \pi.$$
After re-arranging the sums,
$$ f_{n,\lambda}(u) = \sum_{j=0}^n \sum_{\alpha =0}^{n-j} \lambda^{n-j-\alpha} \sum_{\pi \in \Pi_{j,\alpha}} \sum_{k=0}^j (-1)^{j-k} {{n}\choose{k}} S(n-k-1,n-j-\alpha) \sum_{\beta \in Z_{j,\alpha,k}}  P_{\beta,\pi} \pi.$$
In accordance with Definition \ref{coefficient} set
$$ C_\pi = \sum_{k=0}^j (-1)^{j-k} {{n}\choose{k}} S(n-k-1,n-j-\alpha) \sum_{\beta \in Z_{j,\alpha,k}}  P_{\beta,\pi}$$
so that
$$ f_{n,\lambda}(u) = \sum_{j=0}^n \sum_{\alpha =0}^{n-j} \lambda^{n-j-\alpha} \sum_{\pi \in \Pi_{j,\alpha}} C_\pi \pi.$$
By Lemma \ref{no-constants}, there are no terms of degree zero in the expansion of $f_{n,\lambda}(u)$, so we can start the outer sum at $j=1$.\end{proof}

\subsection{An alternate proof of the first K-L identity}

\noindent Fix a positive integer $n$ and a complex number $\lambda$. We showed in Lemma \ref{pi=lambda^alphau^j} that if $u' = \lambda u$ and $\pi \in \Pi_{j,\alpha}$, then $\pi = \lambda^\alpha u^j$. The K-L polynomial $f_{n,\lambda}(u)$ then becomes
$$f_{n,\lambda}(u) = \sum_{j=1}^n \sum_{\alpha=0}^{n-j} \lambda^{n-j-\alpha} \sum_{\pi \in \Pi_{j,\alpha}} C_\pi \lambda^\alpha u^j = \sum_{j=1}^n \left( \sum_{\alpha=0}^{n-j} \sum_{\pi \in \Pi_{j,\alpha}} C_\pi \right) \lambda^{n-j} u^j. $$
\begin{notation}Let $j$ be a positive integer and define\label{C*j}
$$ C^*_j = \sum_{\alpha=0}^{n-j} \sum_{\pi \in \Pi_{j,\alpha}} C_\pi.$$
\end{notation}
\noindent Substituting this definition into our expression for $f_{n,\lambda}(u)$ gives
$$ f_{n,\lambda}(u) = \sum_{j=1}^n C_j^* \lambda^{n-j} u^j.$$
If for each $1 \le j \le n$ it holds that $C_j^* = 0$, then $f_{n,\lambda}(u) = 0$.\bigskip 

\begin{lemma} Fix $1 \le j \le n$. Let $C^*_j$ be as defined in Notation \ref{C*j}. Then
$$ C_j^* = \sum_{k=0}^j (-1)^{j-k} {{n}\choose{k}} \sum_{\alpha = 0}^{n-j} W(j,\alpha,k) S(n-k-1,n-j-\alpha).$$\end{lemma}

\begin{proof} From the proof of Theorem \ref{expanded-form},
$$ C_\pi = \sum_{k=0}^j (-1)^{j-k} {{n}\choose{k}} S(n-k-1,n-j-\alpha) \sum_{\beta \in Z_{j,\alpha,k}}  P_{\beta,\pi}.$$
Therefore,
\begin{align*}
C^*_j & = \sum_{\alpha=0}^{n-j} \sum_{\pi \in \Pi_{j,\alpha}} C_\pi \\
& = \sum_{\alpha=0}^{n-j} \sum_{\pi \in \Pi_{j,\alpha}} \sum_{k=0}^j (-1)^{j-k} {{n}\choose{k}} S(n-k-1,n-j-\alpha) \sum_{\beta \in Z_{j,\alpha,k}}  P_{\beta,\pi}.
\end{align*}
We re-arrange the sums:
$$ C_j^* = \sum_{k=0}^j (-1)^{j-k} {{n}\choose{k}} \sum_{\alpha=0}^{n-j} S(n-k-1,n-j-\alpha) \sum_{\beta \in Z_{j,\alpha,k}} \sum_{\pi \in \Pi_{j,\alpha}} P_{\beta,\pi}.$$
From the definition of the weight function (see Definition \ref{weight}),
$$ \sum_{\beta \in Z_{j,\alpha,k}} \sum_{\pi \in \Pi_{j,\alpha}} P_{\beta,\pi} = W(j,\alpha, k),$$
so
$$ C_j^* = \sum_{k=0}^j (-1)^{j-k} {{n}\choose{k}} \sum_{\alpha = 0}^{n-j} W(j,\alpha,k) S(n-k-1,n-j-\alpha).\qedhere$$
\end{proof}

\begin{lemma} \label{C*j=0} Fix $1 \le j \le n$. Let $C^*_j$ be as defined in Notation \ref{C*j}. Then
$$ C_j^* = \sum_{k=0}^j (-1)^{j-k} {{n}\choose{k}} \sum_{\alpha = 0}^{n-j} W(j,\alpha,k) S(n-k-1,n-j-\alpha) = 0.$$\end{lemma}

\begin{proof}Recall that Lemma \ref{product-sum-factorial} showed
$$ \sum_{\alpha =0}^n m^{\alpha+1} S(n,n-\alpha) = \frac{(m+n)!}{(m-1)!}. $$
To use Lemma \ref{product-sum-factorial} we must render $S(n-k-1,n-j-\alpha)$ into a form similar to the one that appears in Lemma \ref{product-sum-factorial}, namely
$$ S(n-k-1,n-k-1-\alpha).$$
The goal is achieved by letting $\gamma$ be a new index such that
$$ \alpha = \gamma - j + k + 1.$$
The inner sum over $\alpha$ becomes
$$ \sum_{\gamma = j-k-1}^{n-k-1} W(j,\gamma-j+k+1,k)S(n-k-1,n-k-1-\gamma).$$
If $\gamma < j - k-1$, then
$$ W(j,\gamma-j+k+1,k) = 0 $$
since $\gamma- (j-k-1) < 0$ (see Definition \ref{weight}). We may freely add the zero terms indexed by $0 \le \gamma \le j-k-2$. Our inner sum becomes
$$ \sum_{\gamma = 0}^{n-k-1} W(j,\gamma-j+k+1,k)S(n-k-1,n-k-1-\gamma).$$
We now apply Lemma \ref{weight-as-sum}, which says that there exist rational numbers $A_{m,j}$ such that
$$ W(j,\gamma-j+k+1,k) = \sum_{m=k}^j \frac{m^{m-k}}{(m-k)!} A_{m,j} m^{\gamma-j+k+1}.$$
We simplify the above expression:
$$  W(j,\gamma-j+k+1,k) = \sum_{m=k}^j \frac{m^{m-j}}{(m-k)!} A_{m,j} m^{\gamma + 1}.$$
We obtain
\begin{align*}& \sum_{\gamma=0}^{n-j} \sum_{m=k}^j \frac{m^{m-j}}{(m-k)!}A_{m,j} m^{\gamma+1} S(n-k-1,n-k-1-\gamma) \\
= &  \sum_{m=k}^j \frac{m^{m-j}}{(m-k)!}A_{m,j} \sum_{\gamma=0}^{n-j} m^{\gamma+1} S(n-k-1,n-k-1-\gamma).\end{align*}
Finally, we use Lemma \ref{product-sum-factorial}, which says that
$$ \sum_{\gamma=0}^{n-j} m^{\gamma+1} S(n-k-1,n-k-1-\gamma) = \frac{(n-k+m-1)!}{(m-1)!}.$$
This gives us that
\begin{align*}& \sum_{m=k}^j \frac{m^{m-j}}{(m-k)!}A_{m,j} \sum_{\gamma=0}^{n-j} m^{\gamma+1} S(n-k-1,n-k-1-\gamma)\\ = & \sum_{m=k}^j \frac{m^{m-j}}{(m-k)!}A_{m,j}\frac{(n-k+m-1)!}{(m-1)!}.\end{align*}
Therefore,
$$ C_j^* = \sum_{k=0}^j (-1)^{j-k} {{n}\choose{k}} \sum_{m=k}^j \frac{m^{m-j}}{(m-k)!}A_{m,j}\frac{(n-k+m-1)!}{(m-1)!}.$$
Now $C_j^*$ is expressed mostly in terms of factorials. We re-order the summation:
\begin{align*} C_j^* &= \sum_{k=0}^j (-1)^{j-k} {{n}\choose{k}} \sum_{m=k}^j \frac{m^{m-j}}{(m-k)!}A_{m,j}\frac{(n-k+m-1)!}{(m-1)!} \\
& = \sum_{k=0}^j \sum_{m=k}^j (-1)^{j-k} {{n}\choose{k}}  \frac{m^{m-j}}{(m-k)!}A_{m,j}\frac{(n-k+m-1)!}{(m-1)!} \\
& = \sum_{m=1}^j \sum_{k=0}^m (-1)^{j-k} {{n}\choose{k}}  \frac{m^{m-j}}{(m-k)!}A_{m,j}\frac{(n-k+m-1)!}{(m-1)!}
\end{align*}
We multiply and divide by $(n-1)!$:
\begin{align*} C_j^* &= \sum_{m=1}^j \sum_{k=0}^m (-1)^{j-k} {{n}\choose{k}}  \frac{m^{m-j}}{(m-k)!}A_{m,j}\frac{(n-k+m-1)!}{(m-1)!} \\
& = (n-1)! \sum_{m=1}^j \sum_{k=0}^m (-1)^{j-k} {{n}\choose{k}}  \frac{m^{m-j}}{(m-k)!}A_{m,j}\frac{(n-k+m-1)!}{(m-1)!(n-1)!} \end{align*}
Next, we multiply and divide by $(-1)^m$:
\begin{align*} C_j^* &= (n-1)! \sum_{m=1}^j (-1)^{j-m} \frac{m^{m-j}}{(m-1)!} A_{m,j} \sum_{k=0}^m (-1)^{m-k} {{n}\choose{k}} \frac{(n-k+m-1)!}{(m-k)!(n-1)!} \\
& = (n-1)! \sum_{m=1}^j (-1)^{j-m} \frac{m^{m-j}}{(m-1)!} A_{m,j} \sum_{k=0}^m (-1)^{m-k} {{n}\choose{k}} {{n - 1 + m-k}\choose{m-k}}
\end{align*}
We finally claim that
$$\sum_{k=0}^m (-1)^{-k} {{n}\choose{k}} {{n - 1 + m-k}\choose{m-k}} = 0$$
if $m\ge1$, giving $C_j^* = 0$.

Proposition \ref{negative-binomial} says that
$$(-1)^{m-k} {{n + m -k - 1}\choose{m-k}} = {{-n}\choose{m-k}}$$
so
$$ \sum_{k=0}^m (-1)^{m-k} {{n}\choose{k}} {{n - 1 + m-k}\choose{m-k}} = \sum_{k=0}^m {{n}\choose{k}} {{-n}\choose{m-k}}.$$
Using the binomial formula and series convolution, this can be interpreted as the sum of all $k$-th coefficients of $(1+z)^{n}$ multiplied by the $m-k$-th coefficient of $(1+z)^{-n}$. Equivalently, this is the $m$-th coefficient of $(1+z)^n(1+z)^{-n}$, which is the constant function:
\begin{equation} \sum_{k=0}^m {{n}\choose{k}} {{-n}\choose{m-k}} = [z^m](1+z)^n(1+z)^{-n} = [z^m]1 = \begin{cases} 0 & m \ge 1 \\ 1 & m = 0 \end{cases} \label{coeffs-of-one}\end{equation}
where $[z^m]\phi$ denotes the coefficient on $z^m$ in the series of expansion of some function $\phi$.

Of course, the constant function has no nonzero $z^m$ terms if $m \ge 1$. Therefore, $C^*_j = 0$ as claimed.
\end{proof}

\begin{remark} Recall (see Remark \ref{weight-remark}) that care was taken to ensure that in the definition of $W(j,\alpha,k)$, the $k=0$ case reduces to the $k=1$ case. Therefore, we only consider $m \ge 1$. If $m = 0$, then $(m-1)!$ is undefined, so the claimed expression for $C_j^*$ is nonsense.\end{remark}

\begin{proof}[Proof of Theorem \ref{kl1}] Fix a positive integer $n$ and a complex number $\lambda$. If $u' = \lambda u$, then using Notation \ref{C*j},
$$ f_{n,\lambda}(u) = \sum_{j=1}^n C_j^* \lambda^{n-j} u^j.$$
Lemma \ref{C*j=0} then gives that $C_j^* = 0$ for all $1 \le j \le n$, so $f_{n,\lambda}(u) = 0$.\end{proof}


%% file: ch4j.tex
%
%

\noindent Thanks to Thomas Bellsky for helpful conversations about this chapter. For more information about the decomposition of differential operators and the use of differential algebra to extend these methods to the PDE case, we refer the reader to \cite{Sch}.

\subsection{The linear part of the K-L polynomials}

\noindent We begin by stating some lemmas necessary to establish Theorem \ref{thm5}, which largely concern the first-degree terms of $f_{n,\lambda}(u)$. Recall that in our expanded form for the $f_{n,\lambda}(u)$ given in Theorem \ref{expanded-form} that the index $j$ refers to the degree of a term (see Definition \ref{pi}).

\begin{definition}\label{linear-part}The {\it linear part} of $f_{n,\lambda}(u)$, denoted $f_{n,\lambda}^L(u)$, is the sum of all terms in $f_{n,\lambda}(u)$ where $j=1$, in other words
$$ f_{n,\lambda}^L(u) = \sum_{\alpha=0}^{n-1} \lambda^{n-1-\alpha} \sum_{\pi \in \Pi_{1,\alpha}} C_\pi \pi.$$\end{definition}

\begin{lemma} Let $C_\alpha = C_{u^{(\alpha)}}$. Then
$$ f_{n,\lambda}^L(u) = \sum_{\alpha=0}^{n-1} \lambda^{n-1-\alpha} C_\alpha u^{(\alpha)}.$$\end{lemma}

\begin{proof} 
This follows from the definition of $\pi \in \Pi_{1,\alpha}$
since we require $\alpha_1 + \cdots + \alpha_j = \alpha_1 = \alpha$ in this setting.\end{proof}

\begin{lemma}\label{Calpha}Where the function $S(n,\alpha)$ is defined as in Theorem \ref{product-sum},$$ C_\alpha = nS(n-2,n-1-\alpha) - S(n-1,n-1-\alpha).$$\end{lemma}

\begin{proof} The $C_\alpha$ are attached to terms in $f_{n,\lambda}(u)$ whose terms have degree $j=1$. The operator
$$ \prod_{m=0}^{n-k-1} (\partial - u + \lambda u) $$
does not decrease the degree of the term. Therefore, we know these terms may only arise from $k=1$ term,
$${{n}\choose{1}}(\partial- u)\cdots (\partial-u+(n-2)\lambda)u,$$
and the $k=0$ term,
$${{n}\choose{0}}(\partial-u)\cdots(\partial-u+(n-1)\lambda)1.$$
Fix $\alpha$. The term
$${{n}\choose{1}}(D- u)\cdots (D-u+(n-2)\lambda)u  $$
contributes $u^{(\alpha)}$ when the left-most operator and $\alpha-1$ of the remaining operators act by differentiating $u$. There are $n-2$ operators other than the left-most one, so $n-2-(\alpha - 1) = n-1-\alpha$ of these operators act by multiplication by an integer and $\lambda$. (None of these operators will act as multiplication by $-u$; otherwise, we would arrive at a term of degree $j > 1$.) Any option gives a product of $n-1-\alpha$ integers between $1$ and $n-2$. The sum of all these is
$$S(n-2,n-1-\alpha).$$
We multiply by ${{n}\choose{1}} = n$ to complete this part of the coefficient.

Likewise, the term
$${{n}\choose{0}}(\partial-u)\cdots(\partial-u+(n-1)\lambda)1$$
must act by multiplication by $-u$ exactly once to contribute a term of degree $j=1$. The remaining $n$ operators must contribute $\alpha$ derivatives. This leaves $n-1-\alpha$ operators that act as multiplication by an integer and $\lambda$. Adding together all of the options yields
$$ S(n-1,n-1-\alpha).$$
Multiplying by ${{n}\choose{0}} = 1$ and adding the contributions of the terms together gives
$$ C_\alpha = nS(n-2,n-1-\alpha) - S(n-1,n-1-\alpha) $$
as needed.
\end{proof}

\noindent The next lemma provides us with a generating function for the $C_\alpha$.

\begin{lemma}\label{h(z)} Let $n \ge 2$ and
$$ h_{n-1}(z) = (n-1)(1-z)\prod_{m=1}^{n-2}(1+mz).$$
Then $h_{n-1}(z)$ is a generating function for the numbers $C_{n-1-\alpha}$, or 
$$ h_{n-1}(z) = \sum_{\alpha=0}^n C_{n-1-\alpha}z^\alpha.$$\end{lemma}

\begin{proof}The expression for $h_{n-1}(z)$ looks like the expression for $g_{n-1}(z)$ given in Lemma \ref{g(z)}, which recall is written in summation form as
$$ g_{n-1}(z) = \sum_{\alpha = 0}^{n-1} S(n-1,\alpha)z^\alpha.$$
We will begin with the expression of the $C_{\alpha}$ determined in Lemma \ref{Calpha},
$$ C_\alpha = nS(n-2,n-1-\alpha) - S(n-1,n-1-\alpha),$$
and build it into something that allows us to leverage $g_{n-1}(z)$.
Replace $\alpha$ with $n-1-\alpha$:
$$ C_{n-1-\alpha} = nS(n-2,\alpha) - S(n-1,\alpha).$$
Next, multiply both sides by $z^\alpha$ and sum from $\alpha =0$ to $\alpha = n-1$.
\begin{align*}
& C_{n-1-\alpha}z^\alpha = nS(n-2,\alpha)z^\alpha - S(n-1,\alpha)z^\alpha. \\
& \sum_{\alpha=0}^{n-1} C_{n-1-\alpha}z^\alpha = n\sum_{\alpha=0}^{n-1}S(n-2,\alpha)z^\alpha - \sum_{\alpha=0}^{n-1} S(n-1,\alpha)z^\alpha.
\end{align*}
Because $S(n-2,n-1) = 0$, we re-index the left sum in the following way:
$$\sum_{\alpha=0}^{n-1} C_{n-1-\alpha}z^\alpha = n\sum_{\alpha=0}^{n-2}S(n-2,\alpha)z^\alpha - \sum_{\alpha=0}^{n-1} S(n-1,\alpha)z^\alpha.$$
By Lemma \ref{g(z)},
$$h_{n-1}(z) = ng_{n-2}(z) - g_{n-1}(z).$$
Now, we use the product expression for $g_{n-1}(z)$:
\begin{align*}
h_{n-1}(z) & = n\prod_{m=1}^{n-2}(1 + mz) -\prod_{m=1}^{n-1}(1 + mz) \\
& = [n - (1 + (n-1)z]\prod_{m=1}^{n-2}(1 + mz) \\
& = (n-1)(1-z)\prod_{m=1}^{n-2}(1 + mz). \qedhere
\end{align*}
\end{proof}
\noindent The following corrolary will be necessary in the next section:
\begin{cor}\label{h(z)-roots}The roots of $h_{n-1}(z)$ are all rational; they are $z = 1, -1, -\frac{1}{2}, \ldots, -\frac{1}{n-2}$.\end{cor}

\subsection{Proof of the theorem}

\begin{definition}\label{roots-of-unity} Let $q$ be a positive integer. A complex number $\zeta$ is called a $q${\it-th root of unity} if $\zeta^q = 1$, and is called {\it primitive} if $\zeta^n \neq 1$ for all $2 \le n \le q-1$.\end{definition}

\begin{theorem}\label{thm5}Let $u$ be a smooth function, $\lambda$ a nonzero complex number, and $m \ge 3$ an integer such that $u^{(m)} = \lambda^m u$. Let $n$ be an integer greater than $2$. If $f^L_{n,\lambda}(u) = 0$, then $u' = \lambda u$ or $u'' = \lambda^2 u$.\end{theorem}

\begin{proof} By the hypothesis that $u^{(m)} = \lambda^m u$, we can solve the characteristic equation
$$ y^m = \lambda^m$$
to determine that all solutions to $u^{(m)} = \lambda^m u$ have the form
$$ u(x) = \sum_{r = 0}^{m-1} \beta_r e^{\lambda \zeta^r x},$$
where $\zeta$ is a primitive $m$-th root of unity and $\beta_0, \ldots, \beta_{m-1}$ are complex scalars.

Recall from Definition \ref{linear-part} that the linear part of $f_{n,\lambda}(u)$ is given by
$$ f{n,\lambda}^L(u) = \sum_{\alpha = 0}^{n-1} C_\alpha u^{(\alpha)}.$$
We will substitute
$$ u(x) = \sum_{r = 0}^{m-1} \beta_r e^{\lambda \zeta^r x}$$
into $f_{n,\lambda}^L(u) = 0$ and show that the only nonzero scalars are $\beta_0$ and, if $m$ is even, $\beta_{\frac{m}{2}}$.
\begin{align*}f^L(u) & = \sum_{\alpha = 0}^{n-1} \lambda^{n-1-\alpha} C_\alpha \partial^\alpha \left( \sum_{r=0}^{m-1} \beta_r e^{\lambda \zeta^r x} \right) \\
&= \sum_{r=0}^{m-1} \sum_{\alpha = 0}^{n-1} \lambda^{n-1-\alpha} C_\alpha \beta_r \lambda^\alpha \zeta^{r\alpha} e^{\lambda \zeta^r x}\\
& = \sum_{r=0}^{m-1} \lambda^{n-1} \beta_r \left( \sum_{\alpha=0}^{n-1} \zeta^{r\alpha} C_\alpha \right) e^{\lambda \zeta^r x}.
\end{align*}
Recall that the sum
$$ \sum_{\alpha=0}^{n-1} \zeta^{r\alpha} C_\alpha = \sum_{\alpha=0}^{n-1} \bar{\zeta}^{-r\alpha} C_\alpha = \bar{\zeta}^{r(1-n)} \sum_{\alpha=0}^{n-1} \bar{\zeta}^{r(n-1-\alpha)} C_\alpha$$
is exactly 
$$\bar{\zeta}^{r(1-n)} h_{n-1}(\bar{\zeta}^{-r}) = \zeta^{r(n-1)}h_{n-1}(\zeta^r),$$
where $h_{n-1}(z)$ is the generating polynomial defined in Lemma \ref{h(z)}. Finally, we arrive at
$$ \sum_{r=0}^{m-1} \lambda^{n-1} \zeta^{r(n-1)} \beta_r h_{n-1}(\zeta^r) e^{\lambda \zeta^r x} = 0,$$
which is a relation of linear dependence among the distinct functions $e^{\lambda \zeta^r x}$. The $e^{\lambda \zeta^r x}$ are linearly independent, and the only way this relation can hold is if for all $r$,
$$ \lambda^{n-1} \zeta^{r(n-1)} \beta_r h_{n-1}(\zeta^r) = 0. $$
The parameter $\lambda$ is nonzero, as is $\zeta$, and the only way $h_{n-1}(\zeta^r)$ can be zero is if $\zeta^r$ is rational (see Corollary \ref{h(z)-roots}). This implies that if $h_{n-1}(\zeta^r)=0$, then $r$ must be zero, or if $m$ is even, $\frac{m}{2}$. In either case, we have $h_{n-1}(1)=0$ or $h_{n-1}(-1)=0$ and we are forced to conclude that the only nonzero scalars in the expression for $u$ are $\beta_0$ and  $\beta_{\frac{m}{2}}$ if $\frac{m}{2}$ is an integer. Therefore if $m$ is odd, then
$$ u(x) = \beta_0 e^{\lambda x},$$
so $u' = \lambda u$. If $m$ is even, then
$$ u(x) = \beta_0 e^{\lambda x} + \beta_{\frac{m}{2}} e^{-\lambda x},$$
so $u'' = \lambda^2 u$.
\end{proof}

\subsection{The roots of the linear part}

\noindent Not only do our methods reveal that $\partial^m u - \lambda^m u$ is not a factor of the decomposition of the linear part of the K-L polynomial, they actually give an entire decomposition. Observe that we can write the linear part in the following way:
$$ f^L_{n,\lambda}(u) = \sum_{\alpha=0}^{n-1} \lambda^{n-1-\alpha} C^\alpha u^{(\alpha)} = \left(\sum_{\alpha=0}^{n-1} \lambda^{n-1} C_\alpha \partial^\alpha \right)(u).$$
The operator acting on $u$ looks like it may be an instance of the generating function $h_{n-1}(z)$. The next lemma will prove that this is in fact the case.

\begin{lemma}\label{decomp}Let $\lambda$ be a complex number and $n$ be a positive integer. Let $C_\alpha$ be as found in Lemma \ref{Calpha}. Let $u$ be a smooth function. Then
$$ f_{n,\lambda}(u) = (n-1)(\partial - \lambda)\left(\prod_{a=1}^{n-2}(\partial + a\lambda)\right)(u).$$\end{lemma}

\begin{proof}Expand the operator in the right-hand side of the above equation:
\begin{align*}
& (n-1)(\partial - \lambda)\prod_{a=1}^{n-2}(\partial + a\lambda) \\
= & \; n\partial \prod_{a=1}^{n-2} (\partial + a\lambda) - (\partial + (n-1)\lambda) \prod_{a=1}^{n-2} (\partial + a\lambda) \\
= & \; n\partial \prod_{a=1}^{n-2} (\partial + a\lambda) - \prod_{a=1}^{n-1} (\partial + a\lambda).
\end{align*}

\noindent For example, expanding
$$ \prod_{a=1}^{n-2} (\partial + a\lambda)$$
gives a sum of powers of the operator $\partial$ times powers of $\lambda$ times some integer. To obtain $\partial^\alpha$, we must multiply by $a\lambda$ a total of $n-2-\alpha$ times and add over all possibilities, giving
$$ S(n-2,n-2-\alpha)\lambda^{n-2-\alpha}\partial^\alpha.$$
Therefore,
\begin{align*}
& n\partial \prod_{a=1}^{n-2} (\partial + a\lambda) - \prod_{a=1}^{n-1} (\partial + a\lambda) \\
= & \; n\sum_{\alpha=0}^{n-2} S(n-2,n-2-\alpha)\lambda^{n-2-\alpha}\partial^{\alpha+1} - \sum_{\alpha=0}^{n-2} S(n-2,n-2-\alpha)\lambda^{n-2-\alpha}\partial^\alpha.
\end{align*}
Using the fact that $S(n-2,n-1)=0$, we re-index:
\begin{align*}
& n\sum_{\alpha=0}^{n-1} S(n-1,n-1-\alpha)\lambda^{n-1-\alpha}\partial^\alpha - \sum_{\alpha=0}^{n-1} S(n-1,n-1-\alpha)\lambda^{n-1-\alpha}\partial^\alpha. \\
= & \; \sum_{\alpha=0}^{n-1} \left(nS(n-2,n-1-\alpha) - S(n-1,n-1-\alpha)\right)\lambda^{n-1-\alpha}\partial^\alpha \\
= & \; \sum_{\alpha=0}^{n-1} \lambda^{n-1-\alpha} C_\alpha \partial^\alpha
\end{align*}
since $C_\alpha = nS(n-2,n-1-\alpha) - S(n-1,n-1-\alpha)$ (see Lemma \ref{Calpha}).
\end{proof}

\begin{theorem}Let $\lambda$ be a complex number and $n$ be a positive integer. All roots of $f_{n,\lambda}^L(u)$ are of the form
$$ u(x) = \beta_0 e^{\lambda x} + \sum_{a=1}^{n-2} \beta_a e^{-\lambda ax}$$
where $\beta_0, \beta_1, \ldots, \beta_{n-2}$ are complex.\end{theorem}

\begin{proof} Since $f^L_{n,\lambda}(u)$ is a linear differential polynomial, the theorem is equivalent to saying that the solution space of $f_{n,\lambda}^L(u) = 0$ is spanned by $$\setlist{e^{\lambda x}, e^{-\lambda x}, e^{-2\lambda x}, \ldots, e^{-(n-2)\lambda x}}.$$ This is a set of $n-1$ functions and $f^L_{n,\lambda}(u)=0$ is a $(n-1)$-th order differential equation. Furthermore, the elements in the set are distinct exponential functions and so are linearly independent. Therefore, the set is big enough to span the solution space. It remains to show that all of its members are roots of the linear K-L polynomial.

Arbitrarily choose $a$ from $\setlist{-1, 1, 2, \ldots, n-2}$. The choice of $a$ corresponds to both an operator $(\partial + a\lambda)$ in the decomposition given in Lemma \ref{decomp} and a function $e^{-a\lambda x}$ in the proposed spanning set: in fact,
$$ (\partial + a\lambda)e^{-a\lambda x} = -a\lambda e^{-a\lambda x} + a\lambda e^{-a\lambda x} = 0.$$
Observe that since the operators in Lemma \ref{decomp} are defined in terms of $\partial$ and complex scalars, they commute. Therefore, there exists an operator $\psi_a$ such that
$$ f_{n,\lambda}^L(e^{-a\lambda x}) = \psi_a \circ (\partial + a\lambda) e^{-a\lambda x} = 0.$$
This gives that all of the members of $\setlist{e^{\lambda x}, e^{-\lambda x}, e^{-2\lambda x}, \ldots, e^{-(n-2)\lambda x}}$ are roots of the linear K-L polynomial. Since the set is large enough and is linearly independent, it spans the solution space of $f_{n,\lambda}^L(u) = 0$, and any root of $f_{n,\lambda}^L(u)$ has the form
$$ u(x) = \beta_0 e^{\lambda x} + \sum_{a=1}^{n-2} \beta_a e^{-\lambda ax}$$
as claimed. \end{proof}

\subsection{Loosening the $\lambda =0$ restriction}

\noindent The reader may wonder, when the original results allow for nonzero $\lambda$---and when a feature of the alternate proof given in the third chapter is that we treat $\lambda = 0$ in the same case---why this chapter and the following one require $\lambda$ to be nonzero.

Notice that key to this proof was writing $u(x)$ as a sum of exponential functions as a consequence of the fact that $u^{(m)} = \lambda^m u$. If $\lambda =0$, then $u(x)$ is instead a polynomial. Furthermore, if $\lambda = 0$, then the linear part
$$ f^L(u) = \sum_{\alpha = 0}^{n-1} \lambda^{n-1-\alpha} C_\alpha u^{(\alpha)} $$
might vanish without providing us any information about $u$. 

If the index $\alpha$ is less than $n-1$, then that term will be multiplied by zero. Thus all that remains is
$$ f^L(u) = C_{n-1}u^{(n-1)} = 0 $$
which, again, reveals no information when $n-1 > m$.

This leads us to state the following theorem, which provides us with information on a potential pattern of $f_{n,0}(u)$ that vanish if $u^{(m)} = 0$.

\begin{theorem}\label{dmu=0} Let $m$ be a positive integer such that $u^{(m)} = 0$. Suppose there exists a smallest positive integer $k$ such that $f_{k+1,\lambda}(u) =0$. Then $u^{(k)} = 0$.\end{theorem}

\begin{proof} If $k \ge m$, then successive differentiation yields $u^{(k)} = 0$ as claimed. Suppose instead that $k < m$. As previously shown,
$$ f^L_{k+1,0}(u) = \sum_{\alpha = 0}^{k} \lambda^{k-\alpha} C_\alpha u^{(\alpha)} $$
has disappearing terms for every index except $\alpha = k$. Then
$$ f^L_{k+1,0}(u) = C_{k}u^{(k)} = 0.$$
It remains to show that $C_{k} \neq 0$. Recall from Lemma \ref{Calpha} that an expression for $C_\alpha$ is
$$ C_\alpha = (k+1)S(k-1,k-\alpha) - S(k,k-\alpha).$$
Then
$$ C_{k} = (k+1)S(k-1,0) - S(k,0) = k + 1 - 1 = k.$$
We required $k \ge 1$, so $u^{(k)} = 0$.
\end{proof}

\noindent This reveals a relationship between a vanishing derivative of $u$ with the smallest order and any pattern of $f_{n,0}(u) = 0$ satisfied by $u$:

\begin{cor} If $m$ is the smallest integer such that $u^{(m)} = 0$, then any pattern of $f_{n,0}(u)$ that vanish must start at $n \ge m+1$.\end{cor}

\subsection{Conclusion}

\noindent This chapter demonstrated the following facts about the roots of the Kuchment-Lvin polynomials:

\begin{enumerate}
	\item If $u$ is a root of a linear K-L polynomial such that $u^{(m)} = \lambda^mu$ and $\lambda \neq 0$, then either $u' = \lambda u$ or $u'' = \lambda^2 u$.
	\item Therefore, the only patterns of vanishing linear K-L polynomials induced by $u^{(m)} = \lambda^m u$ with nonzero $\lambda$ occur for $m=1$ and $m=2$.
	\item If $\lambda = 0$, then any pattern of vanishing linear K-L polynomials must begin with the polynomial of index at least $m+1$.
\end{enumerate}

\noindent Certainly there is more to learn about the case where $\lambda =0$. To obtain more results, we suspect one must move away from the linear part of $f_{n,0}(u)$ to terms of higher degree. Also, we have only made claims about the linear part of the polynomial: what can we infer about the whole polynomial?

 Additionally, we would like to extend Theorem \ref{thm5} to the entire K-L polynomial. If $u$ is a smooth function and $m$ is an integer such that $m \ge 3$ and $u^{(m)} = \lambda^m u$, then $u = \displaystyle\sum_{r=0}^{m-1} \beta_r e^{\lambda \zeta^r x}$. Therefore $f_{n,\lambda}(u)$ is a polynomial in the functions $e^{\lambda x}, \ldots, e^{\lambda \zeta^{m-1}x}$.

The Lindemann-Weierstrass theorem \cite{Lang} states that if $\gamma_0, \ldots, \gamma_{m-1}$ are linearly independent algebraic numbers over $\mathbf{Q}$, then $e^{\gamma_0 x}, \ldots, e^{\gamma_{m-1} x}$ are algebraically independent over $\mathbf{C}$, {\it i.e.} any polynomial in $m$ variables with complex coefficients satisfied by $e^{\gamma_0 x}, \ldots, e^{\gamma_{m-1}x}$ must be identically zero.

Therefore, if $\lambda, \lambda \zeta, \ldots, \lambda \zeta^{m-1}$ were linearly independent over $\mathbf{Q}$---which they are not---and $u = \displaystyle\sum_{r=0}^{m-1} \beta_r e^{\lambda \zeta^r x}$, then $f_{n,\lambda}(u)$ must be identically zero when viewed as a polynomial of these exponential functions over $\mathbf{C}$. This would force $f_{n,\lambda}^L(u)$ to be identically zero, which would allow us to extend Theorem \ref{thm5} to the whole polynomial.

Of course,
\begin{align*}
& \lambda + \lambda \zeta + \cdots + \lambda \zeta^{m-1} \\
= & \; \lambda \left( 1 + \zeta + \cdots + \zeta^{m-1} \right) = 0, 
\end{align*} so something else must be done.